\documentclass[a4paper, 11pt]{article}
\RequirePackage{fullpage}

\usepackage{multicol}
\usepackage{amsthm,amsmath,amssymb}
\usepackage{booktabs,array}
\usepackage{subcaption}
\usepackage[usenames,dvipsnames]{xcolor}
\usepackage[pdftex,breaklinks,colorlinks,
citecolor={BlueViolet}, linkcolor={Blue},urlcolor=Maroon]{hyperref}
\usepackage{tikz}

\usepackage{xfrac}

\usepackage{charter,eulervm}
\usepackage{enumerate}
\usepackage[final,expansion=alltext,protrusion=true]{microtype}
\usepackage[left]{lineno} 
\usepackage[colorinlistoftodos,prependcaption,textsize=tiny]{todonotes}

\theoremstyle{plain} 
\newtheorem{theorem}{Theorem}[section]
\newtheorem{lemma}[theorem]{Lemma}
\newtheorem{corollary}[theorem]{Corollary}
\newtheorem{proposition}[theorem]{Proposition}

\newtheorem{definition}[theorem]{Definition}

\tikzstyle{filled vertex}  = [{circle,draw=blue,fill=black!50,inner sep=1pt}]  
\tikzstyle{empty vertex}  = [{circle, draw, fill = white, inner sep=1.5pt}]

\title{Ricci Curvature of Strong Product Graphs}
\author{
	Guiqiang Mou\thanks{School of Computer Science and Engineering, Central South University, Changsha, China. {guiqiang.mou@csu.edu.cn}}
}
\date{}

\begin{document}
	\maketitle
	\captionsetup[figure]{labelfont={bf},name={Fig.},labelsep=period}
	
	\begin{abstract}
		
		We establish for the first time the explicit curvature formulas for the horizontal and vertical edges of the strong product of two regular graphs. 
		We complement this result with showing that there does not exist an analogous formula for the curvatures of diagonal edges except for a special case,
		and providing a sharp lower bound for them in terms of the curvatures of the factors. 
		This gives the curvature formulas for all the edges of the product of a complete graph and a regular graph.
		We also present an accessible and simpler proof of the curvature formulas for all the edges of the Cartesian product of two regular graphs, originally established by Lin, Lu, and Yau [2011].
		
	\end{abstract}
	
	\section{Introduction}
	
	Ricci curvature is a fundamental concept in differential geometry and serves as a crucial analytical tool in various areas of mathematics and theoretical physics.  Given its importance, over the past forty years, the notion of Ricci curvature has been generalized to various metric spaces~\cite{bakry2006diffusions,chung1996logarithmic,erbar2012ricci,forman2003bochner,lin2010ricci}, including graphs. In particular, in 2009, Ollivier~\cite{ollivier2009ricci} defined a notion of Ricci curvature of Markov chains on arbitrary metric spaces, which is referred to as Ricci curvature.  In 2011, Lin, Lu, and Yau~\cite{lin2011ricci} introduced a modification of Ricci curvature in a limiting form, which is referred to as Lin--Lu--Yau curvature. Both notions can be defined on graphs. 
	
	Both concepts have attracted considerable attention from researchers in graph theory and geometry.
	In the context of graphs, one of the main lines of study is the computation of curvatures of graphs.	
	Although it can be reduced to solving a linear programming, to develop computational formulas for curvatures of graphs is a meaningful and challenging task. 
	In this aspect, 
	Jost and Liu~\cite{jost2014ollivier} first built the exact formula for Ricci curvature for graphs with girth at least 6, especially for trees.  
	Subsequently, Bhattacharya and Mukherjee~\cite{bhattacharya2015exact} extended this result and obtained the exact formulas for Ricci curvature for bipartite graphs and for graphs with girth at least 5. 
	Bourne et al.~\cite{bourne2018ollivier} studied the Ricci curvature of graphs as a function of the chosen idleness, and proved that this idleness function is concave and piecewise linear with at most three linear parts, and at most two linear parts in the case of a regular graph.
	In 2020, Bonini et al.~\cite{bonini2020condensed} proved an exact formula for Lin--Lu--Yau curvature on strongly regular graphs.  
	Very recently, Hehl~\cite{hehl2024ollivier} derived precise expressions for the Lin--Lu--Yau curvature and Ricci curvature with vanishing idleness of an edge of a regular graph using the concept of optimal assignment, which will play a critical role in the proofs of our results. 
	
	In the original paper, Lin et al.~\cite{lin2011ricci} studied the Cartesian product of two regular graphs.  Cushing et al.~\cite{cushing2021curvatures} considered Ricci curvature of the strong product of two regular graphs.  They proved lower curvature bounds for horizontal and vertical edges of~$G\boxtimes H$ in terms of curvatures of the factors when~$G$ and~$H$ are regular graphs. 
	In addition, Zhang and Fang~\cite{zhang2023spectrum} proved that the Lin--Lu--Yau curvature for strong product of two regular graphs can be (upper and lower) bounded by the curvatures of the factors.  
	The definitions will be deferred to the next section.
	
	\begin{theorem}[\cite{cushing2021curvatures}] \label{th:lowerbound}
		Let~$G$ and~$H$ be two regular graphs with vertex degrees~$d_G$ and~$d_H$, respectively.  For any vertex~$x$ of~$G$ and any pair of adjacent vertices~$y_1$ and~$y_2$ of $H$,
		\begin{align*}
			\kappa_{\rm{LLY}}\left( (x, y_1), (x, y_2) \right) &\geq \frac{d_H(d_G + 1)}{d_{G\boxtimes H}}\kappa_{\rm{LLY}}(y_1, y_2),
			\\
			\kappa_0((x, y_1), (x, y_2)) &\geq \frac{d_H(d_G + 1)}{d_{G\boxtimes H}}\kappa_0(y_1, y_2), 
		\end{align*}
		where~$d_{G\boxtimes H} = d_G + d_H + d_Gd_H$ is the vertex degree of~$G\boxtimes H$.
	\end{theorem}
	
	We prove that the equality in first inequality of Theorem~\ref{th:lowerbound} always holds.  Indeed, we give the exact curvature formulas for horizontal and vertical edges, and one special class of diagonal edges of~$G\boxtimes H$. 
	
	\begin{theorem} \label{th:strong}
		Let~$G$ and~$H$ be two regular graphs with vertex degrees~$d_G$ and~$d_H$, respectively.  For any pair of vertices~$x_{1}, x_{2} \in V(G)$ with~$N_{G}[x_1] = N_{G}[x_2]$ and any pair of adjacent vertices~$y_1$ and~$y_2$ of $H$,
		\begin{align}
			\kappa_{\rm{LLY}}((x_1, y_1), (x_2, y_2)) & = \frac{d_H(d_G + 1)}{d_{G\boxtimes H}}\kappa_{\rm LLY}(y_1, y_2), 
			\label{eq:LLYstrong}
			\\
			\kappa_0((x_1, y_1), (x_2, y_2)) & = \frac{d_Gd_H}{d_{G\boxtimes H}}\kappa_{\rm LLY}(y_1, y_2) + \frac{d_H}{d_{G\boxtimes H}}\kappa_0(y_1, y_2),
			\label{eq:k0strong}
		\end{align}
		where~$d_{G\boxtimes H} = d_G + d_H + d_Gd_H$ is the vertex degree of~$G\boxtimes H$.
	\end{theorem}
	
	The proof of Theorem~\ref{th:strong} can be found in Section~\ref{sec:strong product}.
	Theorem~\ref{th:strong} subsumes a special class of diagonal edges of $G\boxtimes H$. 
	For diagonal edge with endpoints~$(x_1, y_1)$ and~$(x_2, y_2)$ such that $N_G[x_1] \neq N_G[x_2]$ and $N_H[y_1] \neq N_H[y_2]$, 
	it is natural to wonder whether there exists an exact formula for~$\kappa_{\rm{LLY}}((x_1, y_1), (x_2, y_2))$ which is purely in terms of~$\kappa_{\rm{LLY}}(x_1, x_2)$, $\kappa_{\rm{LLY}}(y_1, y_2)$, $d_G$, and $d_H$. 
	We give a simple example to explain that no such a formula exists. Further we prove a sharp lower bound for $\kappa_{\rm{LLY}}((x_1, y_1), (x_2, y_2))$ in terms of ~$\kappa_{\rm{LLY}}(x_1, x_2)$ and $\kappa_{\rm{LLY}}(y_1, y_2)$.
	
	\begin{theorem} \label{th:noformula}
		Let~$G$ and~$H$ be two regular graphs with vertex degrees~$d_G$ and~$d_H$, respectively.  For any pair
		of adjacent vertices~$x_1$ and~$x_2$ of $G$ and any pair of adjacent vertices~$y_1$ and~$y_2$ of $H$, where $N_G[x_1] \neq N_G[x_2]$ and $N_H[y_1] \neq N_H[y_2]$,
		\[
			\kappa_{\rm{LLY}}((x_1, y_1), (x_2, y_2)) \geq 
			\frac{d_H(d_G + 1)}{d_{G\boxtimes H}}\kappa_{\rm LLY}(y_1, y_2) + \frac{d_G(d_H + 1)}{d_{G\boxtimes H}}\kappa_{\rm LLY}(x_1, x_2) - X,  
		\] 	
		where 
		\[
		X = 1 - \frac{|R_{x_1}(x_1, x_2)||R_{y_1}(y_1, y_2)| - 1}{d_{G\boxtimes H}}.
		\]
	\end{theorem}

	Lin et al.~\cite{lin2011ricci}  have established the curvature formulas for all the edges of the Cartesian product of two regular graphs.
	In Section~\ref{sec:Cartesian product}, we will present a simpler proof of Theorem~\ref{th:cartesian} utilizing the result of Hehl~\cite{hehl2024ollivier}.

	\begin{theorem}[\cite{lin2011ricci}] \label{th:cartesian}
		Let~$G$ and~$H$ be two regular graphs with vertex degrees~$d_G$ and~$d_H$, respectively. For any vertex~$x$ of~$G$ and any pair of adjacent vertices~$y_1$ and~$y_2$ of $H$,
		\begin{align*}
			\kappa_{\rm{LLY}}((x, y_1), (x, y_2)) & = \frac{d_H}{d_{G\Box H}}\kappa_{\rm{LLY}}(y_1, y_2),
			\\ 
			\kappa_0((x, y_1), (x, y_2)) & = \frac{d_H}{d_{G\Box H}}\kappa_0(y_1, y_2),
		\end{align*}
		where~$d_{G\Box H} = d_G + d_H$ is the vertex degree of~$G\Box H$.
	\end{theorem}
	
	\section{Preliminaries} \label{sec:preliminary}
	
	Let~$G = (V, E)$ be a graph with vertex set~$V$ and edge set~$E$. We say~$G$ is~\emph{locally finite} if for any~$x\in V$, its degree~$d_x$ is finite. 
	For an edge $xy\in E$, we define 
	\begin{align*}
		\bigtriangleup(x, y) & = N(x)\cap N(y), \\
		R_x(x, y) & = N(x)\setminus (\bigtriangleup(x, y)\cup \{y\}), \\
		R_y(x, y) & = N(y)\setminus (\bigtriangleup(x, y)\cup \{x\}).
	\end{align*}

	A~\emph{probability measure} on~$V$ is a map~$\mu: V \to [0, 1]$ such that~$\sum_{x\in V}\mu(x) = 1$.
	Define the probability measure~$\mu_x^\alpha$ for~$x\in V$ and~\emph{idleness}~$\alpha\in [0, 1]$ by
	\begin{align*}
		\mu_x^\alpha(y) =
		\begin{cases}
			\alpha & \text{if $y = x$,}\\
			\frac{1 - \alpha}{d_x} & \text{if $y\in N(x)$,} \\
			0  & \text{otherwise.}
		\end{cases}
	\end{align*}
	
	\begin{definition}[Wasserstein distance]
		Let~$G = (V, E)$ be a locally finite graph,~$\mu_1$ and~$\mu_2$ be two probability measures on~$V$. The Wasserstein distance between~$\mu_1$ and~$\mu_2$ is defined as 
		\[
		W(\mu_1, \mu_2) = \inf_{\pi\in \Pi(\mu_1, \mu_2)}\sum_{x\in V}\sum_{y\in V}d(x, y)\pi(x, y), 
		\]
		where~$\Pi(\mu_1, \mu_2)$ denotes the set of all maps~$\pi: V\times V \to [0, 1]$ such that 
		\[
		\sum_{y\in V}\pi(x, y) = \mu_1(x), \quad \sum_{x\in V}\pi(x, y) = \mu_2(y). 
		\]
	\end{definition}
	
	Intuitively, to understand~$W(\mu_1, \mu_2)$, one can image~$\mu_1$ and~$\mu_2$ as two piles of earth. Then~$W(\mu_1, \mu_2)$ can be taken as the minimum cost of transporting~$\mu_1$ to~$\mu_2$. 
	
	\begin{definition}[Ricci curvature]
		Let~$G = (V, E)$ be a locally finite graph. For two distinct vertices~$x, y\in V$,~\textit{$\alpha$-Ricci curvature} is defined by 
		\[
		\kappa_{\alpha}^{G}(x,y) = 1 - \frac{W(\mu_{x}^{\alpha}, \mu_{y}^{\alpha})}{d(x,y)}.
		\]
		
	\end{definition}
	
	Lin, Lu, and Yau~\cite{lin2011ricci} first investigated the property of the~\emph{idleness function}~$\alpha\to \kappa_{\alpha}^{G}(x,y)$. They proved that the limit~$\lim_{\alpha\to 1}\frac{\kappa_{\alpha}^G(x,y)}{1 - \alpha}$ exists. Based on this result, they introduced a modification of the Ricci curvature as follows.

	\begin{definition}[Lin--Lu--Yau curvature]
		Let~$G = (V, E)$ be a locally finite graph. For two distinct vertices~$x, y\in V$,~\textit{Lin--Lu--Yau curvature} is defined by 
		\[
		\kappa_{\rm{LLY}}^{G}(x,y) = \lim_{\alpha\to 1}\frac{\kappa_{\alpha}^{G}(x,y)}{1 - \alpha}.
		\]
	\end{definition}
	
	Let us remark that although Ricci curvature and Lin--Lu--Yau curvature are defined for any two distinct vertices in~$V$, this work focuses on pairs of adjacent vertices in $V$.
	Hereafter if the graph is clear from the context, we will omit the superscripts from~$\kappa_{\rm{LLY}}^{G}$ and~$\kappa_{\alpha}^{G}$. 
	
	Recall that the~\emph{Birkhoff polytope}~$\mathcal{B}_n$ is the convex polytope in~$\mathbb{R}^{n\times n}$ whose points are the~$n\times n$ matrices whose entries are nonnegative real numbers and whose rows and columns each add up to 1. 
	The~\emph{Birkhoff's theorem}~\cite{birkhoff1946three} states that the extreme points of~$\mathcal{B}_n$ are the~\emph{permutation matrices}, that is, matrices with exactly one entry 1 in each row and each column and 0 elsewhere.    
	Using Birkhoff's theorem, 
	Hehl~\cite{hehl2024ollivier} 
	reduced the computation of the Wasserstein distance~$W(\mu_x^{\alpha}, \mu_y^{\alpha})$ to an assignment problem between~$R_x(x, y)$ and~$R_y(x, y)$, and derived precise expressions for~$\kappa_{\rm{LLY}}(x,y)$ and~$\kappa_{0}(x,y)$ for any edge $xy$ with $d_x = d_y$, which are given in Theorems~\ref{th:formula} and~\ref{th:differnce}, respectively.

	\begin{definition}[Optimal assignment]
		Let~$G=(V,E)$ be a locally finite graph. Let~$U, W \subset V$ with~$\vert U \vert = \vert W \vert < \infty$. We call a bijection~$\phi: U \to W$ an~\textit{assignment} between~$U$ and~$W$. Denote by~$\mathcal{A}_{UW}$ the set of all assignments between~$U$ and~$W$. We call~$\phi \in \mathcal{A}_{UW}$ an \textit{optimal assignment} between~$U$ and~$W$ if
		\begin{equation*}
			\sum_{z \in U} d(z,\phi(z)) = \inf_{\psi \in \mathcal{A}_{UW}} \sum_{z \in U} d(z,\psi(z)).
		\end{equation*}
	\end{definition}
	
	Let $x, y\in V$ be of equal degree with $xy\in E$.
	For convenience, we will use~$\mathcal{A}_{xy}$ to denote the set of all assignments between~$R_x(x, y)$ and~$R_y(x, y)$, and~$\mathcal{O}_{xy}$ to denote the set of all optimal assignments between~$R_x(x, y)$ and~$R_y(x, y)$.

	\begin{theorem}[\cite{hehl2024ollivier}] \label{th:formula}
		Let~$G = (V, E)$ be a locally finite graph, and
		$x, y\in V$ be of equal degree $d$ with $xy\in E$. 
		Then
		\begin{align}
			\kappa_{\rm LLY}(x, y) = \frac{1}{d}(d + 1 - \inf_{\phi\in \mathcal{A}_{xy}}\sum_{z\in R_x(x, y)}d(z, \phi(z))).
			\label{eq:LLY}
		\end{align}
	\end{theorem}
	
	\begin{theorem}[\cite{hehl2024ollivier}] \label{th:differnce}
		Let~$G = (V, E)$ be a locally finite graph, and 
		$x, y\in V$ be of equal degree $d$ with $xy\in E$. If~$|\bigtriangleup(x, y)| < d - 1$, then 
		\begin{align}
			\kappa_0(x, y) = \kappa_{\rm{LLY}}(x, y) - \frac{1}{d}(3 - \sup_{\phi\in \mathcal{O}_{xy}}\sup_{z\in R_x(x, y)}d(z, \phi(z)))
			\label{eq:k01}.
		\end{align}
		If~$|\bigtriangleup(x, y)| = d - 1$, then 
		\begin{align}
			\kappa_0(x, y) = \kappa_{\rm{LLY}}(x, y) - \frac{2}{d}.
			\label{eq:k02}
		\end{align}
	\end{theorem}	
	
	In the~\emph{strong product} of two graphs~$G$ and~$H$, denoted as~$G\boxtimes H$, the vertex set is~$V(G)\times V(H)$, and two distinct vertices~$(x_1, y_1)$ and~$(x_2, y_2)$ are adjacent if and only if~$x_2\in N_G[x_1]$ and~$y_2\in N_H[y_1]$. 
	In other words, one of the following three conditions holds true:
	(i)~$y_1y_2\in E(H)$ and $x_1=x_2$, (ii)~$x_1x_2\in E(G)$ and $y_1=y_2$, or (iii)~$x_1x_2\in E(G)$ and $y_1y_2\in E(H)$.
	These conditions lead to three kinds of edges in product  graph, which are called, respectively, \emph{horizontal}, \emph{vertical}, and~\emph{diagonal} edges. 
	In the~\emph{Cartesian product} of two graphs~$G$ and~$H$, denoted as~$G\Box H$, the vertex set is~$V(G)\times V(H)$, and~$(x_1, y_1)$ and~$(x_2, y_2)$ are adjacent if and only if
	the condition (i) or (ii) holds. 
	
	To present the proofs better, for any two adjacent vertices $(x_1, y_1)$ and $(x_2, y_2)$ of equal degree in~$G\boxtimes H$ or~$G\Box H$, we define the following: 
	\begin{align*}
		\text{OPT}_{G*H} & = \inf_{\phi\in \mathcal{A}_{(x_1, y_1)(x_2, y_2)}}\sum_{z\in R_{(x_1, y_1)}((x_1, y_1), (x_2, y_2))}d(z, \phi(z)), \\
		\text{OPT}_H & = \inf_{\phi\in \mathcal{A}_{y_1y_2}}\sum_{z\in R_{y_1}(y_1, y_2)}d(z, \phi(z)), \\
		\text{MAX}_{G*H} & = \sup_{\phi\in \mathcal{O}_{(x_1, y_1)(x_2, y_2)}}\sup_{z\in R_{(x_1, y_1)}((x_1, y_1), (x_2, y_2))}d(z, \phi(z)), \\ 
		\text{MAX}_H & = \sup_{\phi\in \mathcal{O}_{y_1y_2}}\sup_{z\in R_{y_1}(y_1, y_2)}d(z, \phi(z)).
	\end{align*}
	Here~$G*H$ refers to~$G\boxtimes H$ or~$G\Box H$, and the subscripts of~$\text{OPT}$ and~$\text{MAX}$ indicate on which graph we consider the value of the expression on the right. Similarly, we can define $\text{OPT}_G$ and $\text{MAX}_G$.
	
	\section{Ricci curvature of strong product} \label{sec:strong product}
	
	To prove Theorem~\ref{th:strong}, according to Theorems~\ref{th:formula} and~\ref{th:differnce}, we wish to build the relation between~$\text{OPT}_{G\boxtimes H}$ and~$\text{OPT}_ H$, and the relation between~$\text{MAX}_{G\boxtimes H}$ and~$\text{MAX}_H$.
	
	
	The following theorem is well-known; e.g., it can be easily derived from Hall's Theorem.
	
	\begin{theorem}[Folklore] \label{co:regualrbi}
		Let~$G = (U, V)$ be a regular bipartite graph which may have multiple edges. Then $G$ has a perfect matching. 
	\end{theorem}	
	
	The next proposition follows from the definition of strong product graphs.
	
	\begin{proposition} \label{pro:distancestrong}
		Let~$G$ and~$H$ be two connected graphs. For any two vertices~$(x_1, y_1)$ and~$(x_2, y_2)$ in~$G\boxtimes H$,
		\[
		d\left((x_1, y_1), (x_2, y_2)\right) = \max\left\{d(x_1, x_2), d(y_1, y_2) \right\}. 
		\]
	\end{proposition}
	\begin{proof}
		Without loss of generality, assume~$d(x_1, x_2)\leq d(y_1, y_2)$. Let~$u_0u_1\ldots u_s$ be a shortest path between~$x_1$ and~$x_2$ in~$G$ where~$u_0 = x_1$ and~$u_s = x_2$. 
		Let~$v_0v_1\ldots v_t$ be a shortest path between~$y_1$ and~$y_2$ in~$H$, where~$v_0 = y_1$ and~$v_t = y_2$.  
		By definition
		\[
		(u_0, v_0)(u_1, v_1)\ldots (u_s, v_s)(u_s, v_{s +1 })\ldots (u_s, v_t)
		\]
		is a path between~$(x_1, y_1)$ and~$(x_2, y_2)$ in~$G\boxtimes H$. 
		Thus, we have 
		\begin{align}
			d((x_1, y_1), (x_2, y_2)) \leq d(y_1, y_2) = \max\{d(x_1, x_2), d(y_1, y_2)\}.
			\label{ineq:disstrleq}
		\end{align}
		
		On the other hand, let~$P$ be a shortest path between~$(x_1, y_1)$ and~$(x_2, y_2)$ in~$G\boxtimes H$. 
		By definition the projection of~$P$ on~$G$, denoted as $P_G$, is a path between~$x_1$ and~$x_2$ in~$G$, 
		and the projection of~$P$ on~$H$, denoted as $P_H$, is a path between~$y_1$ and~$y_2$ in~$H$. 
		Thus, we have 
		\begin{align}
			d((x_1, y_1), (x_2, y_2)) \geq \max\{|P_G|, |P_H|\} \geq \max\{d(x_1, x_2), d(y_1, y_2)\}, 
			\label{ineq:disstrgeq}
		\end{align}
		where~$|P_G|$ and~$|P_H|$ denote the lengths of paths~$P_G$ and~$P_H$, respectively.
		Combining inequalities (\ref{ineq:disstrleq}) and (\ref{ineq:disstrgeq}) finishes the proof. 
	\end{proof}
	
	The next two lemmas establish the relations between ~$\text{OPT}_{G\boxtimes H}$ and~$\text{OPT}_ H$, and between~$\text{MAX}_H$ and~$\text{MAX}_{G\boxtimes H}$, respectively. 
	In product graph, the subgraph induced by 
	the vertex set~$\{(x, y)\mid y\in V(H)\}$ is called an~\emph{$H$-fiber}, and is denoted by~$^{x}H$. 
	
	\begin{lemma} \label{le:infstrong}
		Let~$G$ and~$H$ be two regular graphs with vertex degrees~$d_G$ and~$d_H$, respectively. 
		Let $x_1, x_2\in V(G)$ with~$N_{G}[x_1] = N_{G}[x_2]$, and $y_1y_2\in E(H)$. Then 
		\[
		\text{OPT}_{G\boxtimes H} = (d_G + 1)\text{OPT}_H.
		\]
	\end{lemma}	
	\begin{proof}
		By definition of strong product, under the conditions in lemma, we have 
		\begin{align*}
			R_{(x_1, y_1)}((x_1, y_1), (x_2, y_2)) & = N_G[x_1]\times R_{y_1}(y_1, y_2),
			\\
			R_{(x_2, y_2)}((x_1, y_1), (x_2, y_2)) & = N_G[x_1]\times R_{y_2}(y_1, y_2).
		\end{align*}
		Let~$\phi_1\in \mathcal{O}_{(x_1, y_1)(x_2, y_2)}$, and~$\phi_2\in \mathcal{O}_{y_1y_2}$.
		On the one hand, one can construct an assignment~$\phi'_1\in \mathcal{A}_{(x_1, y_1), (x_2, y_2)}$ by taking the assignment~$\phi_2$ in each~$H$-fiber, $^{x}H$ where~$x\in N_G[x_1]$, that is, $\phi'_1((x, y)) = (x, \phi_2(y))$ for any $(x, y)\in R_{(x_1, y_1)}((x_1, y_1), (x_2, y_2))$.  
		By Proposition~\ref{pro:distancestrong}, 
		\begin{align*}
			d((x, y), \phi'_1((x, y))) = d((x, y), (x, \phi_2(y))) = d(y, \phi_2(y)).
		\end{align*}
		It follows that
		\begin{align}
			\sum_{z\in R_{(x_1, y_1)}((x_1, y_1), (x_2, y_2))}d(z, \phi_1(z)) & \leq
			\sum_{z\in R_{(x_1, y_1)}((x_1, y_1), (x_2, y_2))}d(z, \phi'_1(z)) \nonumber \\
			& = \sum_{x\in N_G[x_1]}\sum_{y\in R_{y_1}(y_1, y_2)}d((x, y), \phi'_1(x, y)) \nonumber \\
			& = \sum_{x\in N_G[x_1]}\sum_{y\in R_{y_1}(y_1, y_2)}d(y, \phi_2(y)) \nonumber \\
			& = (d_G + 1) \sum_{z\in R_{y_1}(y_1, y_2)}d(z, \phi_2(z)).
			\label{ineq:infleqstrong}
		\end{align}
		
		On the other hand, we claim that~$R_{(x_1, y_1)}((x_1, y_1), (x_2, y_2))$ 
		can be partitioned into~$(d_G + 1)$ groups of equal size, denoted as~$X_i$,~$i = 1, 2, \ldots, d_G + 1$ 
		such that for each group~$X_i$, $\{y\mid (x, y)\in X_i\} = R_{y_1}(y_1, y_2)$,
		and the map~$\phi_{X_i}: R_{y_1}(y_1, y_2)\to R_{y_2}(y_1, y_2)$ defined as
		\[
		\phi_{X_i}(y) = y' \quad \text{if } \exists(x, y)\in X_i \text{ such that }\phi_1((x, y)) = (x', y')
		\]
		forms an assignment~$\phi_{X_i}\in \mathcal{A}_{y_1y_2}$.
		Recall that 
		\[
		R_{(x_1, y_1)}((x_1, y_1), (x_2, y_2)) = N_G[x_1]\times R_{y_1}(y_1, y_2).
		\] 
		Consider the auxiliary graph $H'$, which is a bipartite multigraph on vertex sets~$R_{y_1}(y_1, y_2)$ and~$R_{y_2}(y_1, y_2)$, 
		and
		there exists one edge between~$y\in R_{y_1}(y_1, y_2)$ and~$y'\in R_{y_2}(y_1, y_2)$ in $H'$ if and only if there exists one vertex $(x, y)\in N_G[x_1]\times R_{y_1}(y_1, y_2)$ such that~$\phi_1((x, y)) = (x', y')$. 
		Since $\phi_1$ is a bijection between $N_G[x_1]\times R_{y_1}(y_1, y_2)$ and $N_G[x_1]\times R_{y_2}(y_1, y_2)$. 
		Obviously $H'$ is $(d_G + 1)$-regular.
		By Theorem~\ref{co:regualrbi}, $H'$ has a perfect matching $M$. 
		Due to the construction of~$H'$, for each edge $yy'\in M$,
		there exists $(x, y)\in (N_G[x_1]\times R_{y_1}(y_1, y_2))$ such that~$\phi_1((x, y)) = (x', y')$. Let 
		\[
		X' = \{(x, y)\in (N_G[x_1]\times R_{y_1}(y_1, y_2))\mid yy'\in M, \text{ and } \phi_1((x, y)) = (x', y')\}
		\]
		be one group. 
		It should be noted that for an edge~$yy'\in M$, there may exist more than one~$(x, y)\in (N_G[x_1]\times R_{y_1}(y_1, y_2))$ such that~$\phi_1((x, y)) = (x', y')$, and if this is the case, 
		we take only one of them into~$X'$. 
		Remove~$X'$ from~$N_G[x_1]\times R_{y_1}(y_1, y_2)$ and its image~$\phi_1(X')$ from~$N_G[x_1]\times R_{y_2}(y_1, y_2)$, and consider the assignment $\phi_1$ restricted between the remaining sets. 
		Repeat the above argument by $d_G$ times, we will obtain all~$(d_G + 1)$ groups. Thus, the claim holds true.   
		By Proposition~\ref{pro:distancestrong} and the definition of~$\phi_{X_i}$, for each group~$X_i$, and each~$(x, y)\in X_i$,
		\[
		d((x, y), \phi_1((x, y))) \geq d(y, \phi_{X_i}(y)).
		\]
		It follows that
		\begin{align}
			\sum_{z\in R_{(x_1, y_1)}((x_1, y_1), (x_2, y_2))}d(z, \phi_1(z)) & = \sum_{i = 1}^{d_G + 1}\sum_{z\in X_i}d(z, \phi_1(z)) \nonumber \\
			& \geq \sum_{i = 1}^{d_G + 1}\sum_{(x, y)\in X_i}d(y, \phi_{X_i}(y)) \nonumber \\
			& \geq (d_G + 1)\sum_{z\in R_{y_1}(y_1, y_2)}d(z, \phi_2(z)).
			\label{ineq:infgeqstrong}
		\end{align} 
		Combining inequalities (\ref{ineq:infleqstrong}) and (\ref{ineq:infgeqstrong}) finishes the proof.
	\end{proof}
	
	\begin{lemma} \label{le:supstrong}
		Let~$G$ and~$H$ be two regular graphs with vertex degrees~$d_G$ and~$d_H$, respectively. 
		Let $x_1, x_2\in V(G)$ with~$N_{G}[x_1] = N_{G}[x_2]$, and $y_1y_2\in E(H)$. Then 
		\[
		\text{MAX}_{G\boxtimes H} = \text{MAX}_H.
		\]
	\end{lemma}
	\begin{proof}
		Let~$\phi_1\in \mathcal{O}_{(x_1, y_1)(x_2, y_2)}$ such that
		\[
		\sup_{z\in R_{(x_1, y_1)}((x_1, y_1), (x_2, y_2))}d(z, \phi_1(z))
		\]
		reaches the maximum over all optimal assignments~$\phi\in \mathcal{O}_{(x_1, y_1)(x_2, y_2)}$.
		Let~$\phi_2\in \mathcal{O}_{y_1y_2}$ such that
		\[
		\sup_{z\in R_{y_1}(y_1, y_2)}d(z, \phi_2(z))
		\]
		reaches the maximum over all optimal assignments~$\phi\in \mathcal{O}_{y_1y_2}$.
		
		On the one hand, by Lemma~\ref{le:infstrong}, one can construct an optimal assignment~$\phi'_1\in \mathcal{O}_{(x_1, y_1)(x_2, y_2)}$ by taking the assignment~$\phi_2$ in each~$H$-fiber, $^{x}H$ where~$x\in N_G[x_1]$. 
		By Proposition~\ref{pro:distancestrong}, for any $(x, y)\in R_{(x_1, y_1)}((x_1, y_1), (x_2, y_2))$,
		\begin{align*}
			d((x, y), \phi'_1((x, y))) = d((x, y), (x, \phi_2(y))) = d(y, \phi_2(y)).
		\end{align*}
		It follows that
		\begin{align}
			\sup_{z\in R_{(x_1, y_1)}((x_1, y_1), (x_2, y_2))}d(z, \phi_1(z)) & \geq \sup_{z\in R_{(x_1, y_1)}((x_1, y_1), (x_2, y_2))}d(z, \phi'_1(z)) \nonumber \\
			& = \sup_{z\in R_{y_1}(y_1, y_2)}d(z, \phi_2(z)).
			\label{ine:supgeqstrong}
		\end{align}
		
		On the other hand, by Lemma~\ref{le:infstrong}, we conclude that~$R_{(x_1, y_1)}((x_1, y_1), (x_2, y_2))$ can be partitioned into~$(d_G + 1)$ groups such that for each group~$X$, 
		the map~$\phi_X$ associated with~$\phi_1$
		forms an optimal assignment~$\phi_X\in \mathcal{O}_{y_1, y_2}$ (recall the proof of Lemma~\ref{le:infstrong}).
		In addition, for each group~$X$, and each~$(x, y)\in X$,
		\begin{align}
			d((x, y), \phi_1((x, y))) = d(y, \phi_X(y)).
			\label{eq:dxy}
		\end{align}
		Let~$z' = (x', y')\in R_{(x_1, y_1)}((x_1, y_1), (x_2, y_2))$ such that 
		\begin{align*}
			d(z', \phi_1(z')) = \sup_{z\in R_{(x_1, y_1)}((x_1, y_1), (x_2, y_2))}d(z, \phi_1(z)).
		\end{align*}
		Denote by~$X'$ the group which contains~$z'$. 
		By equation (\ref{eq:dxy}),
		\begin{align*}
			d(y', \phi_{X'}(y')) = d(z', \phi_1(z')).
		\end{align*}
		It follows that 
		\begin{align}
			\sup_{z\in R_{y_1}(y_1, y_2)}d(z, \phi_2(z)) & \geq \sup_{z\in R_{y_1}(y_1, y_2)}d(z, \phi_{X'}(z)) \nonumber \\
			& = \sup_{z\in R_{(x_1, y_1)}((x_1, y_1), (x_2, y_2))}d(z, \phi_1(z)). 
			\label{ine:supleqstrong}
		\end{align}
		Combining inequalities (\ref{ine:supgeqstrong}) and ($\ref{ine:supleqstrong}$) completes the proof.
	\end{proof}

	We are now ready to prove Theorem \ref{th:strong}. 
	
	\begin{proof}[Proof of Theorem~\ref{th:strong}]
		We first prove equation (\ref{eq:LLYstrong}). 
		By Lemma \ref{le:infstrong}, 
		\begin{align}
			\text{OPT}_{G\boxtimes H} & = (d_G + 1)\text{OPT}_H.
			\label{eq:XGHstrong}
		\end{align}
		Applying equation (\ref{eq:LLY}) on graph~$H$, we get
		\begin{align}
			\text{OPT}_H = d_H - d_H\kappa_{\rm{LLY}}(y_1, y_2) + 1.
			\label{eq:XHstrong}
		\end{align}
		Combine equations (\ref{eq:LLY}), (\ref{eq:XGHstrong}), and (\ref{eq:XHstrong}),
		\begin{align*}
			\kappa_{\rm LLY}((x_1, y_1), (x_2, y_2)) & = \frac{1}{d_{G\boxtimes H}}(d_{G\boxtimes H} + 1 - \text{OPT}_{G\boxtimes H})\\
			& = \frac{d_H(d_G + 1)}{d_{G\boxtimes H}}\kappa_{\rm LLY}(y_1, y_2). 	  
		\end{align*}
		Next we prove equation (\ref{eq:k0strong}).
		
		Case 1. $|\bigtriangleup((x_1, y_1), (x_2, y_2))| < d_{G\boxtimes H} - 1$, i.e.,~$|\bigtriangleup(y_1, y_2)| < d_H - 1$. 
		By Lemma~\ref{le:supstrong}, 
		\begin{align}
			\text{MAX}_{G\boxtimes H} = \text{MAX}_H.
			\label{eq:YGH=YHstrong}
		\end{align}
		Applying equation (\ref{eq:k01}) on~$H$, we get
		\begin{align}
			\text{MAX}_H = d_H(\kappa_0(y_1, y_2) - \kappa_{\rm{LLY}}(y_1, y_2)) + 3.
			\label{eq:YHstrong}
		\end{align}    
		Combine equations (\ref{eq:LLYstrong}), (\ref{eq:k01}), (\ref{eq:YGH=YHstrong}), and (\ref{eq:YHstrong}), 
		\begin{align*}
			\kappa_0((x_1, y_1), (x_2, y_2)) & = \frac{d_H(d_G + 1)}{d_{G\boxtimes H}}\kappa_{\rm LLY}(y_1, y_2) - \frac{1}{d_{G\boxtimes H}}(3 - \text{MAX}_{G\boxtimes H})\\
			& = \frac{d_Gd_H}{d_{G\boxtimes H}}\kappa_{\rm LLY}(y_1, y_2) + \frac{d_H}{d_{G\boxtimes H}}\kappa_0(y_1, y_2). 
		\end{align*}
		
		Case 2. $|\bigtriangleup((x_1, y_1), (x_2, y_2))| = d_{G\boxtimes H} - 1$, i.e.,~$|\bigtriangleup(y_1, y_2)| = d_H - 1$. Applying equation (\ref{eq:k02}) on $H$, we get 
		\begin{align}
			\kappa_0(y_1, y_2) = \kappa_{\rm{LLY}}(y_1, y_2) - \frac{2}{d_H}.
			\label{eq:k0H}
		\end{align}
		Combine equations (\ref{eq:LLYstrong}), (\ref{eq:k02}) and (\ref{eq:k0H}), 
		\begin{align*}
			\kappa_0((x_1, y_1), (x_2, y_2)) & = \frac{d_H(d_G + 1)}{d_{G\boxtimes H}}\kappa_{\rm LLY}(y_1, y_2)  - \frac{2}{d_{G\boxtimes H}} \\
			& = \frac{d_Gd_H}{d_{G\boxtimes H}}\kappa_{\rm{LLY}}(y_1, y_2) + \frac{d_H}{d_{G\boxtimes H}}(\kappa_{\rm{LLY}}(y_1, y_2) - \frac{2}{d_H}) \\
			& = \frac{d_Gd_H}{d_{G\boxtimes H}}\kappa_{\rm {LLY}}(y_1, y_2) + \frac{d_H}{d_{G\boxtimes H}}\kappa_0(y_1, y_2).
		\end{align*}
		This concludes the proof.
	\end{proof}
	
	In a connected graph~$G$, all edges~$x_1x_2\in E(G)$ satisfy~$N[x_1] = N[x_2]$ if and only if $G$ is a complete graph.
	Thus, as an immediate consequence of Theorem \ref{th:strong}, we have
	
	\begin{corollary}
		Let $G$ be a complete graph, and $H$ be a regular graph with vertex degree $d_H$.
		For any pair of adjacent vertices $(x_1, y_1)$ and $(x_2, y_2)$ of $G\boxtimes H$,
		\begin{align*}
			\kappa_{\rm{LLY}}((x_1, y_1), (x_2, y_2)) =
			\begin{cases}
				\frac{d_H(d_G + 1)}{d_{G\boxtimes H}}\kappa_{\rm LLY}(y_1, y_2) & \text{if $R_{y_1}(y_1, y_2) \neq \emptyset$,} \\
				\frac{d_{G\boxtimes H} + 1}{d_{G\boxtimes H}} & \text{otherwise.}	
			\end{cases}
		\end{align*}
		and
		\begin{align*}
			\kappa_0((x_1, y_1), (x_2, y_2)) =
			\begin{cases}
				\frac{d_Gd_H}{d_{G\boxtimes H}}\kappa_{\rm LLY}(y_1, y_2) + \frac{d_H}{d_{G\boxtimes H}}\kappa_0(y_1, y_2)  & \text{if $R_{y_1}(y_1, y_2) \neq \emptyset$,} \\ 
				\frac{d_{G\boxtimes H} - 1}{d_{G\boxtimes H}} & \text{otherwise.}	
			\end{cases}		
		\end{align*}
		Here~$d_{G\boxtimes H} = d_G + d_H + d_Gd_H$ is the vertex degree of~$G\boxtimes H$.
	\end{corollary} 
	\begin{proof}
		If $R_{y_1}(y_1, y_2) \neq \emptyset$.
		Since $N_G[x_1] = N_G[x_2]$, 
		the result follows directly from Theorem \ref{th:strong}.
		
		If $R_{y_1}(y_1, y_2) = \emptyset$ or $y_ 1 = y_2$. Since $R_{x_1}(x_1, x_2) = \emptyset$, by definition
		\[
		R_{(x_1, y_1)}((x_1, y_1), (x_2, y_2)) = \emptyset.
		\]
		Thus, $\text{OPT}_{G\boxtimes H} = 0$, and $|\bigtriangleup((x_1, y_1), (x_2, y_2))| = d_{G\boxtimes H} -1$. 
		The result follows from equations (\ref{eq:LLY}) and (\ref{eq:k02}). 
	\end{proof}  
	
	Recall the result of  
	Bourne et al.~\cite{bourne2018ollivier}, which states that the full Ricci idleness function $\kappa_{\alpha}(x, y)$ for two adjacent vertices $x$ and $y$ with $d_y | d_x$ can be fully decided by $\kappa_{\rm{LLY}}(x, y)$, $\kappa_{0}(x, y)$, and $d_x$.
	
	\begin{theorem}[\cite{bourne2018ollivier}]
		Let~$G = (V, E)$ be a locally finite graph, and~$x, y\in V$ with $xy\in E$ and~$d_y | d_x$. Then 
		\begin{align}
			\kappa_{\alpha}(x, y) =
			\begin{cases}
				(d_x\kappa_{\rm{LLY}}(x, y) - (d_x + 1)\kappa_0(x, y))\alpha + \kappa_0(x, y) & \text{if $\alpha\in [0, \frac{1}{d_x + 1}]$,}\\
				(1 - \alpha)\kappa_{\rm{LLY}}(x, y) & \text{if $\alpha\in [\frac{1}{d_x + 1}, 1]$.} 
			\end{cases}
			\label{eq:function}
		\end{align}
	\end{theorem}
	
	Substituting equations (\ref{eq:LLYstrong}) and (\ref{eq:k0strong}) into equation (\ref{eq:function}), we get 
	
	\begin{corollary}
		Let~$G$ and~$H$ be two regular graphs with vertex degrees~$d_G$ and~$d_H$, respectively. 
		Let $x_1, x_2\in V(G)$ with~$N_{G}[x_1] = N_{G}[x_2]$, and $y_1y_2\in E(H)$.
		If~$\alpha\in [0, \frac{1}{d_{G\boxtimes H} + 1}]$, then
		\begin{align*}
			\kappa_{\alpha}((x_1, y_1), (x_2, y_2))  = &  d_H(\kappa_{\rm{LLY}}(y_1, y_2) - \kappa_0(y_1, y_2))\alpha 
			+ \frac{d_H}{d_{G\boxtimes H}}(d_G\kappa_{\rm{LLY}}(y_1, y_2) \\
			& +  \kappa_0(y_1, y_2))(1 - \alpha).
		\end{align*}
		If $\alpha\in [\frac{1}{d_{G\boxtimes H} + 1}, 1]$, then
		\[
		\kappa_{\alpha}((x_1, y_1), (x_2, y_2)) = \frac{d_H(d_G + 1)}{d_{G\boxtimes H}}\kappa_{\rm LLY}(y_1, y_2)(1 - \alpha). 
		\] 
	\end{corollary}	
	
	In what follows, we give an example to show that in $G\boxtimes H$, for diagonal edge with endpoints $(x_1, y_1)$ and $(x_2, y_2)$ such that $N_G[x_1] \neq N_G[x_2]$ and $N_H[y_1] \neq N_H[y_2]$,  
	$\kappa_{\rm{LLY}}((x_1, y_1), (x_2, y_2))$ cannot be uniquely determined by~$\kappa_{\rm{LLY}}(x_1, x_2)$, $\kappa_{\rm{LLY}}(y_1, y_2)$, $d_G$, and $d_H$.  
	As a consequence, no exact formula for~$\kappa_{\rm{LLY}}((x_1, y_1), (x_2, y_2))$ purely in terms of~$\kappa_{\rm{LLY}}(x_1, x_2)$, $\kappa_{\rm{LLY}}(y_1, y_2)$, $d_G$, and $d_H$ exists.
	
	Let $G = C_4$, $H_1$ and $H_2$ be the 3-regular graphs with $y_1y_2$ being an edge, depicted in Figure \ref{Fig}. 
	Let $x_1x_2$ be any one edge of $G$.
	By equation (\ref{eq:LLY}), it can be computed that $\kappa_{\rm{LLY}}^{H_1}(y_1, y_2) = \kappa_{\rm{LLY}}^{H_2}(y_1, y_2) = 0$;
	however, $\kappa_{\rm{LLY}}^{G\boxtimes H_1}((x_1, y_1), (x_2, y_2)) = -1/11$, $\kappa_{\rm{LLY}}^{G\boxtimes H_2}((x_1, y_1), (x_2, y_2)) = 0$ ($OPT_{G\boxtimes H_1} = 13$, $OPT_{G\boxtimes H_2} = 12$). 
	
	\begin{figure}[htb]
		\centering
		\begin{subfigure}[b]{0.3\linewidth}
			\centering
			\begin{tikzpicture}[scale=0.5]
				\node[scale=0.5, shape=circle, minimum size=0.1cm, draw] (v1) at ({18}:2) {};
				\node[scale=0.5, shape=circle, minimum size=0.1cm, draw] (v2) at ({90}:2) {};
				\node[scale=0.5, shape=circle, minimum size=0.1cm, draw] (v3) at ({162}:2) {};
				\node[scale=0.5, shape=circle, minimum size=0.1cm, draw, label=below:$y_1$] (v4) at ({234}:2) {};
				\node[scale=0.5, shape=circle, minimum size=0.1cm, draw, label=below:$y_2$] (v5) at ({306}:2) {};
				\node[scale=0.5, shape=circle, minimum size=0.1cm, draw] (v6) at ({18}:0.8) {};
				\node[scale=0.5, shape=circle, minimum size=0.1cm, draw] (v7) at ({90}:1) {};
				\node[scale=0.5, shape=circle, minimum size=0.1cm, draw] (v8) at ({162}:0.8) {};
				\draw (v5)--(v1)--(v2)--(v3)--(v4)--(v4)--(v5);
				\draw (v5)--(v6)--(v7)--(v8)--(v4);
				\draw (v6)--(v1); 
				\draw (v7)--(v2);
				\draw (v8)--(v3); 	
			\end{tikzpicture}
			\caption{$H_1$}
		\end{subfigure}
		\quad \quad
		\begin{subfigure}[b]{0.3\linewidth}
			\centering
			\begin{tikzpicture}[scale=0.5]
				\node[scale=0.5, shape=circle, minimum size=0.1cm, draw] (v1) at ({0}:2) {};
				\node[scale=0.5, shape=circle, minimum size=0.1cm, draw] (v2) at ({60}:2) {};
				\node[scale=0.5, shape=circle, minimum size=0.1cm, draw] (v3) at ({120}:2) {};
				\node[scale=0.5, shape=circle, minimum size=0.1cm, draw] (v4) at ({180}:2) {};
				\node[scale=0.5, shape=circle, minimum size=0.1cm, draw, label=below:$y_1$] (v5) at ({240}:2) {};
				\node[scale=0.5, shape=circle, minimum size=0.1cm, draw, label=below:$y_2$] (v6) at ({300}:2) {};
				\node[scale=0.5, shape=circle, minimum size=0.1cm, draw] (v7) at ({30}:1) {};
				\node[scale=0.5, shape=circle, minimum size=0.1cm, draw] (v8) at ({320}:0.8) {};
				\node[scale=0.5, shape=circle, minimum size=0.1cm, draw] (v9) at ({220}:0.8) {};
				\node[scale=0.5, shape=circle, minimum size=0.1cm, draw] (v10) at ({150}:1) {};
				\draw (v6)--(v1)--(v2)--(v3)--(v4)--(v5)--(v5)--(v6);
				\draw (v2)--(v7)--(v8)--(v9)--(v10)--(v3); 
				\draw (v1)--(v7);
				\draw (v6)--(v8);
				\draw (v5)--(v9);
				\draw (v4)--(v10);	
			\end{tikzpicture}
			\caption{$H_2$}
		\end{subfigure}	
		\caption{The illustrations of graphs $H_1$ and $H_2$.}
		\label{Fig}
	\end{figure}
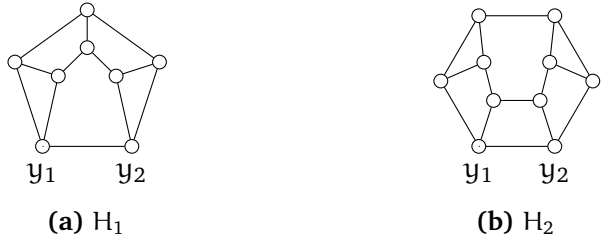

	Next we prove Theorem~\ref{th:noformula}.  
	
	\begin{proof}[Proof of Theorem~\ref{th:noformula}]
		By definition of strong product, we have 
		\begin{align*}
			R_{(x_1, y_1)}((x_1, y_1), (x_2, y_2)) & = 
			(R_{x_1}(x_1, x_2)\times N_H[y_1])\cup (N_G[x_1]\times R_{y_1}(y_1, y_2)),
			\\
			R_{(x_2, y_2)}((x_1, y_1), (x_2, y_2)) & = 
			(R_{x_2}(x_1, x_2)\times N_H[y_2])\cup (N_G[x_2]\times R_{y_2}(y_1, y_2)).   
		\end{align*}
		Let~$\phi_1\in \mathcal{O}_{x_1x_2}$, and~$\phi_2\in \mathcal{O}_{y_1y_2}$. Define $\phi_3\in \mathcal{A}_{(x_1, y_1)(x_2, y_2)}$ as 
		\begin{align*}
			\phi_3(x, y) = \begin{cases}
				(x, \phi_2(y)) & \text{if $(x, y)\in (N_G[x_1]\setminus R_{x_1}(x_1, x_2))\times R_{y_1}(y_1, y_2)$,}\\
				(\phi_1(x), y) & \text{if $(x, y)\in R_{x_1}(x_1, x_2)\times (N_H[y_1]\setminus R_{y_1}(y_1, y_2))$,} \\
				(\phi_1(x), \phi_2(y))  & \text{if $(x, y)\in R_{x_1}(x_1, x_2)\times R_{y_1}(y_1, y_2)$.}
			\end{cases}
		\end{align*}
		By Proposition \ref{pro:distancestrong}, it is easy to get 
		\begin{align*}
			& \sum_{z\in R_{(x_1, y_1)}((x_1, y_1), (x_2, y_2))}d(z, \phi_3(z)) \\
			& \leq \Big((d_G + 1)\sum_{z\in R_{x_1}(x_1, x_2)}d(z, \phi_1(z)) + (d_H + 1)\sum_{z\in R_{y_1}(y_1, y_2)}d(z, \phi_2(z))\Big) - |R_{x_1}(x_1, x_2)||R_{y_1}(y_1, y_2)|
			\\
			& = (d_G + 1)\text{OPT}_H + (d_H + 1)\text{OPT}_G - |R_{x_1}(x_1, x_2)||R_{y_1}(y_1, y_2)|,	
		\end{align*}
		where the first inequality holds since for each $(x, y)\in R_{x_1}(x_1, x_2)\times R_{y_1}(y_1, y_2)$, 
		\begin{align*}
			d((x, y), \phi_3((x, y))) & = d((x, y), (\phi_1(x), \phi_2(y))) \\
			& = \max\left\{d(x, \phi_1(x)), d(y, \phi_2(y)) \right\} \\
			& \leq d(x, \phi_1(x)) + d(y, \phi_2(y)) - 1. 
		\end{align*}
		Thus, we have
		\begin{align}
			\text{OPT}_{G\boxtimes H} & \leq \sum_{z\in R_{(x_1, y_1)}((x_1, y_1), (x_2, y_2))}d(z, \phi_3(z))
			\nonumber \\
			& \leq (d_G + 1)\text{OPT}_H  + (d_H + 1)\text{OPT}_G - |R_{x_1}(x_1, x_2)||R_{y_1}(y_1, y_2)|. 
			\label{ine:upper}
		\end{align}
		Applying equation~(\ref{eq:LLY}) on $G$ and~$H$, we get
		\begin{equation}
			\begin{aligned}
				\text{OPT}_H & = d_H - d_H\kappa_{\rm{LLY}}(y_1, y_2) + 1,
				\\
				\text{OPT}_G & = d_G - d_G\kappa_{\rm{LLY}}(x_1, x_2) + 1.
				\label{eq:GH}
			\end{aligned}
		\end{equation}
		Combining equations (\ref{eq:LLY}), (\ref{eq:GH}) and inequality (\ref{ine:upper}) finishes the proof. 	
	\end{proof}

    Note that from the proof of Theorem \ref{th:noformula}, the lower bound can be attained if $\text{OPT}_G = |R_{x_1}(x_1, x_2)|$, and $\text{OPT}_H = |R_{y_1}(y_1, y_2)|$. For example, let $G = H = C_4$.

	\section{Ricci curvature of Cartesian product} \label{sec:Cartesian product}
	
	The next proposition follows from the definition of Cartesian product graphs.
	
	\begin{proposition}\label{pro:distancecartesian}
		Let~$G$ and~$H$ be two connected graphs. For any two vertices~$(x_1, y_1)$ and~$(x_2, y_2)$ in~$G\Box H$,
		\[
		d((x_1, y_1), (x_2, y_2)) = d(x_1, x_2) + d(y_1, y_2). 
		\]
	\end{proposition}
	\begin{proof}
		Let~$P_G$ be a shortest path between~$x_1$ and~$x_2$ in~$G$, and~$P_H$ be a shortest path between~$y_1$ and~$y_2$ in~$H$. 
		By definition~$(P_G\times \{y_1\})\cup (\{x_2\}\times P_H)$ is a path of length~$d(x_1, x_2) + d(y_1, y_2)$ between~$(x_1, y_1)$ and~$(x_2, y_2)$ in~$G\Box H$. Thus, we have 
		\begin{align}
			d((x_1, y_1), (x_2, y_2)) \leq d(x_1, x_2) + d(y_1, y_2).
			\label{ineq:discarleq}
		\end{align}
		
		On the other hand, let~$Q$ be a shortest path between~$(x_1, y_1)$ and~$(x_2, y_2)$ in~$G\Box H$. 
		By definition the projection of~$Q$ on~$G$, denoted as $Q_G$, is a path between~$x_1$ and $x_2$ in~$G$, and the projection of~$Q$ on~$H$, denoted as $Q_H$, is a path between~$y_1$ and~$y_2$ in~$H$. 
		Thus, we have 
		\begin{align}
			d((x_1, y_1), (x_2, y_2)) \geq |Q_G| + |Q_H| \geq d(x_1, x_2) + d(y_1, y_2), 
			\label{ineq:discargeq}    
		\end{align}
		where~$|Q_G|$ and~$|Q_H|$ denote the lengths of paths~$Q_G$ and~$Q_H$, respectively.
		Combining inequalities (\ref{ineq:discarleq}) and (\ref{ineq:discargeq}) closes the proof. 
	\end{proof}
	
	The next two lemmas establish the relations between ~$\text{OPT}_{G\Box H}$ and~$\text{OPT}_ H$, and between~$\text{MAX}_H$ and~$\text{MAX}_{G\Box H}$, respectively. 
	
	\begin{lemma} \label{le:infcartesian}
		Let~$G$ and~$H$ be two regular graphs with vertex degrees~$d_G$ and~$d_H$, respectively. 
		Let $x_1 = x_2\in V(G)$, and $y_1y_2\in E(H)$. Then
		\[
		\text{OPT}_{G\Box H} = \text{OPT}_H + d_G.
		\]
	\end{lemma}	
	\begin{proof}
		By definition of Cartesian product, we have 
		\begin{align*}
			R_{(x_1, y_1)}((x_1, y_1), (x_1, y_2)) & = (N_G(x_1)\times \{y_1\})\cup (\{x_1\}\times R_{y_1}(y_1, y_2)),
			\\
			R_{(x_1, y_2)}((x_1, y_1), (x_1, y_2)) & = (N_G(x_1)\times \{y_2\})\cup (\{x_1\}\times R_{y_2}(y_1, y_2)).
		\end{align*}
		Let~$\phi_1\in \mathcal{O}_{(x_1, y_1)(x_1, y_2)}$, and~$\phi_2\in \mathcal{O}_{y_1y_2}$. 
		We claim that there exists some~$\phi'_1\in \mathcal{O}_{(x_1, y_1)(x_1, y_2)}$ 
		such that for any~$(x, y_1)\in (N_G(x_1)\times \{y_1\})$, $\phi'_1((x, y_1)) = (x, y_2)$. 
		To prove the claim, due to the optimality of~$\phi'_1$, 
		it suffices to prove that there exists some~$\phi'_1\in \mathcal{O}_{(x_1, y_1)(x_1, y_2)}$ such that~$\phi'_1(N_G(x_1)\times \{y_1\}) = N_G(x_1)\times \{y_2\}$.
		If~$\phi_1$ satisfies this property, then let~$\phi'_1 = \phi_1$ and we are done.
		Suppose that~$\phi_1$ does not.  
		Without loss of generality, assume that there exists only one vertex~$(x', y_1)\in (N_G(x_1)\times \{y_1\})$ such that~$\phi_1((x', y_1))\notin (N_G(x_1)\times \{y_2\})$, i.e., $\phi_1((x', y_1))\in (\{x_1\}\times R_{y_2}(y_1, y_2))$.
		Let~$\phi_1((x', y_1)) = (x_1, y')$. 
		Since~$|N_G(x_1)\times \{y_1\}| = |N_G(x_1)\times \{y_2\}|$. 
		There must exist some~$(x_1, y'')\in (\{x_1\}\times R_{y_1}(y_1, y_2))$ such that~$\phi_1((x_1, y'')) \in (N_G(x_1)\times \{y_2\})$. Let~$\phi_1((x_1, y'')) = (x'', y_2)$. 
		Define~$\phi'_1\in \mathcal{A}_{(x_1, y_1)(x_1, y_2)}$ as
		\begin{align*}
			\phi'_1(z) =
			\begin{cases}
				(x'', y_2) & \text{if $z = (x', y_1)$,}\\
				(x_1, y') & \text{if $z = (x_1, y'')$,} \\
				\phi_1(z)  & \text{otherwise.}
			\end{cases}
		\end{align*}
		By Proposition~\ref{pro:distancecartesian}, 
		we have
		\begin{align}
			& d((x', y_1), \phi_1((x', y_1))) + d((x_1, y''), \phi_1((x_1, y''))) \nonumber \\
			& = d((x', y_1), (x_1, y')) + d((x_1, y''), (x'', y_2)) \nonumber \\
			& = d(x', x_1) + d(y_1, y') + d(x_1, x'') + d(y'', y_2) \nonumber \\
			& = d(y_1, y') + d(y'', y_2) + 2 \nonumber \\
			& = 6.
			\label{eq:6}
		\end{align}
		Here the third equality holds since~$d(x', x_1) = d(x_1, x'') = 1$ ($x', x'' \in N_G(x_1)$), the last equality holds since~$d(y_1, y') = 2$ ($y'\in R_{y_2}(y_1, y_2)$) and $d(y'', y_2) = 2$ ($y''\in R_{y_1}(y_1, y_2)$).
		For~$\phi'_1$, by Proposition~\ref{pro:distancecartesian} we have
		\begin{align}
			& d((x', y_1), \phi'_1((x', y_1))) + d((x_1, y''), \phi'_1((x_1, y''))) \nonumber \\
			& = d((x', y_1), (x'', y_2)) + d((x_1, y''), (x_1, y')) \nonumber \\
			& = d(x', x'') + d(y_1, y_2) + d(y'', y') \nonumber \\
			& = d(x', x'') + 1 + d(y'', y')  \nonumber \\
			& \leq 6,
			\label{ine:leq6}
		\end{align}
		where the last inequality holds since~$d(x', x'')\leq 2$ ($x', x''\in N_G(x_1)$) and~$d(y'', y')\leq 3$ ($y''\in N_H(y_1), y'\in N_H(y_2)$).
		By equation (\ref{eq:6}) and inequality (\ref{ine:leq6}), 
		we have 
		\[
		\sum_{z\in R_{(x_1, y_1)}((x_1, y_1), (x_1, y_2))}d(z, \phi'_1(z))\leq \sum_{z\in R_{(x_1, y_1)}((x_1, y_1), (x_1, y_2))}d(z, \phi_1(z)),
		\]
		which means that $\phi'_1\in \mathcal{O}_{(x_1, y_1)(x_1, y_2)}$. 
		The claim follows from that $\phi'_1(N_G(x_1)\times \{y_1\}) = N_G(x_1)\times \{y_2\}$. 
		
		Now we consider~$\phi'_1$ and~$\phi_2$.
		Since for any~$(x, y_1)\in (N_G(x_1)\times \{y_1\})$, $\phi'_1((x, y_1)) = (x, y_2)$, then
		\begin{align}
			\sum_{(x, y_1)\in (N_G(x_1)\times \{y_1\})}d((x, y_1), \phi'_1((x, y_1))) = \sum_{(x, y_1)\in (N_G(x_1)\times \{y_1\})}1 = d_G. 
			\label{eq:firstpart} 
		\end{align}
		Since~$\phi'_1(\{x_1\}\times R_{y_1}(y_1, y_2)) = \{x_1\}\times R_{y_2}(y_1, y_2)$, due to the optimality of~$\phi'_1$, we have
		\begin{align}
			\sum_{z\in (\{x_1\}\times R_{y_1}(y_1, y_2))}d(z, \phi'_1(z)) = \sum_{z\in R_{y_1}(y_1, y_2)}d(z, \phi_2(z)). 
			\label{eq:secondpart}
		\end{align} 
		Combining equations (\ref{eq:firstpart}) and (\ref{eq:secondpart}) yields
		\begin{align*}
			& \sum_{z\in R_{(x_1, y_1)}((x_1, y_1), (x_1, y_2))}d(z, \phi'_1(z)) \\
			&  = \sum_{z\in (N_G(x_1)\times \{y_1\})}d(z, \phi'_1(z)) + \sum_{z\in (\{x_1\}\times R_{y_1}(y_1, y_2))}d(z, \phi'_1(z)) \\
			& = d_G + \sum_{z\in R_{y_1}(y_1, y_2)}d(z, \phi_2(z)). 
		\end{align*}
		This concludes the proof.
	\end{proof}
	
	\begin{lemma} \label{le:supcartesian}
		Let~$G$ and~$H$ be two regular graphs with vertex degrees~$d_G$ and~$d_H$, respectively. 
		Let $x_1 = x_2\in V(G)$, and $y_1y_2\in E(H)$. If~$R_{y_1}(y_1, y_2)\neq \emptyset$, then 
		\[
		\text{MAX}_{G\Box H} = \text{MAX}_H.
		\]
	\end{lemma}
	\begin{proof}
		Let~$\phi_1\in \mathcal{O}_{(x_1, y_1)(x_1, y_2)}$ such that
		\[
		\sup_{z\in R_{(x_1, y_1)}((x_1, y_1), (x_1, y_2))}d(z, \phi_1(z))
		\]
		reaches the maximum over all optimal assignments~$\phi\in \mathcal{O}_{(x_1, y_1)(x_1, y_2)}$.
		Let~$\phi_2\in \mathcal{O}_{y_1y_2}$ such that 
		\[
		\sup_{z\in R_{y_1}(y_1, y_2)}d(z, \phi_2(z))
		\]
		reaches the maximum over all optimal assignments~$\phi\in \mathcal{O}_{y_1y_2}$.
		
		On the one hand, by Lemma~\ref{le:infcartesian}, one can construct an optimal assignment~$\phi'_1\in \mathcal{O}_{(x_1, y_1)(x_1, y_2)}$ as follows:
		\begin{align*}
			\phi'_1((x, y)) =
			\begin{cases}
				(x, y_2) & \text{if $(x, y)\in (N_G(x_1)\times \{y_1\})$}, \\
				(x, \phi_2(y)) & \text{if $(x, y)\in (\{x_1\}\times R_{y_1}(y_1, y_2))$}. 
			\end{cases}
		\end{align*}
		By Proposition~\ref{pro:distancecartesian}, for each $(x, y)\in (\{x_1\}\times R_{y_1}(y_1, y_2))$,
		\begin{align*}
			d((x, y), \phi'_1((x, y))) = d((x, y), (x, \phi_2(y))) = d(y, \phi_2(y)).
		\end{align*}
		It follows that 
		\begin{align}
			\sup_{z\in R_{(x_1, y_1)}((x_1, y_1), (x_1, y_2))}d(z, \phi_1(z)) & \geq  \sup_{z\in R_{(x_1, y_1)}((x_1, y_1), (x_1, y_2))}d(z, \phi'_1(z)) \nonumber \\
			& = \sup_{z\in R_{y_1}(y_1, y_2)}d(z, \phi_2(z)).
			\label{ine:supgeqcar}
		\end{align}
		
		On the other hand, we claim that there exists some $\phi'_1\in \mathcal{O}_{(x_1, y_1)(x_1, y_2)}$ such that $\phi'_1(N_G(x_1)\times \{y_1\}) = N_G(x_1)\times \{y_2\}$, and such that 
		\[
		\sup_{z\in R_{(x_1, y_1)}((x_1, y_1), (x_1, y_2))}d(z, \phi_1(z)) = \sup_{z\in R_{(x_1, y_1)}((x_1, y_1), (x_1, y_2))}d(z, \phi'_1(z)).
		\] 
		If~$\phi_1$ satisfies this property, then let~$\phi'_1 = \phi_1$ and we are done. 
		Suppose that~$\phi_1$ does not. 
		Without loss of generality, assume that there exists only one vertex~$(x', y_1)\in (N_G(x_1)\times \{y_1\})$ such that~$\phi_1((x', y_1))\notin (N_G(x_1)\times \{y_2\})$, i.e., $\phi_1((x', y_1))\in (\{x_1\}\times R_{y_2}(y_1, y_2))$.  
		Let~$\phi_1((x', y_1)) = (x_1, y')$. 
		Since~$|N_G(x_1)\times \{y_1\}| = |N_G(x_1)\times \{y_2\}|$. 
		There must exist some~$(x_1, y'')\in (\{x_1\}\times R_{y_1}(y_1, y_2))$ such that~$\phi_1((x_1, y'')) \in (N_G(x_1)\times \{y_2\})$. Let~$\phi_1((x_1, y'')) = (x'', y_2)$.
		By Proposition~\ref{pro:distancecartesian}, it is easy to get 
		\begin{align*}
			\sup_{z\in R_{(x_1, y_1)}((x_1, y_1), (x_1, y_2))}d(z, \phi_1(z)) & = d((x', y_1), \phi_1((x', y_1))) \\ 
			& = d((x_1, y''), \phi_1((x_1, y''))) \\
			& = 3.
		\end{align*}
		Define 
		\begin{align*}
			\phi'_1(z) =
			\begin{cases}
				(x'', y_2) & \text{if $z = (x', y_1)$,}\\
				(x_1, y') & \text{if $z = (x_1, y'')$,} \\
				\phi_1(z)  & \text{otherwise.}
			\end{cases}
		\end{align*}
		We have proved in Lemma~\ref{le:supcartesian} that~$\phi'_1\in \mathcal{O}_{(x_1, y_1)(x_1, y_2)}$.   
		Since~$\phi_1, \phi'_1\in \mathcal{O}_{(x_1, y_1)(x_1, y_2)}$, and 
		\begin{align*}
			d((x', y_1), \phi_1((x', y_1)) + d((x_1, y''), \phi_1((x_1, y'')) = 3 +3 = 6.
		\end{align*}
		By Proposition~\ref{pro:distancecartesian}, we have 
		\begin{align*}
			& d((x', y_1), \phi'_1((x', y_1))) + d((x_1, y''), \phi'_1((x_1, y''))) \nonumber \\
			& = d(x', x'') + 1 + d(y'', y')  \nonumber \\
			& = 6.
		\end{align*}
		Since~$d(x', x'')\leq 2$ and~$d(y'', y')\leq 3$. It follows from the above equation that~$d(x', x'') = 2$, and~$d(y'', y') = 3$. By Proposition~\ref{pro:distancecartesian},
		\[
		d((x_1, y''), \phi'_1((x_1, y''))) = d((x_1, y''), (x_1, y')) = d(y'', y') = 3. 
		\]
		Since 
		\[
		\sup_{z\in R_{(x_1, y_1)}((x_1, y_1), (x_1, y_2))}d(z, \phi'_1(z)) = \sup_{z\in R_{(x_1, y_1)}((x_1, y_1), (x_1, y_2))}d(z, \phi_1(z)) = 3
		\]
		and $\phi'_1(N_G(x_1)\times \{y_1\}) = N_G(x_1)\times \{y_2\}$, the claim holds true.  
		
		Now we consider~$\phi'_1$ and~$\phi_2$. 
		Since~$R_{y_1}(y_1, y_2)\neq \emptyset$, and
		for any~$z\in (N_G(x_1)\times \{y_1\})$, $d(z, \phi'_1(z)) = 1$. We have 
		\[
		\sup_{z\in R_{(x_1, y_1)}((x_1, y_1), (x_1, y_2))}d(z, \phi'_1(z)) = \sup_{z\in (\{x_1\}\times R_{y_1}(y_1, y_2))}d(z, \phi'_1(z)). 
		\]
		Due to the optimality of~$\phi'_1$, the assignment~$\phi'_2\in \mathcal{A}_{y_1y_2}$ defined as~$\phi'_2(y) = y'$ if~$\phi'_1(x_1, y) = (x_1, y')$, is optimal. 
		Since for any $y\in R_{y_1}(y_1, y_2)$,
		\[
		d(y, \phi'_2(y)) = d((x_1, y), \phi'_1((x_1, y))). 
		\] 
		It follows that 
		\begin{align}
			\sup_{z\in R_{y_1}(y_1, y_2)}d(z, \phi_2(z)) & \geq \sup_{z\in R_{y_1}(y_1, y_2)}d(z, \phi'_2(z)) \nonumber \\
			& = \sup_{z\in R_{(x_1, y_1)}((x_1, y_1), (x_1, y_2))}d(z, \phi'_1(z)). 
			\label{ine:supleqcar}
		\end{align}
		Combining inequalities (\ref{ine:supgeqcar}) and (\ref{ine:supleqcar}) yields~$\text{MAX}_{G\Box H} = \text{MAX}_H$.  
	\end{proof} 
	
	We are now in a position to prove Theorem~\ref{th:cartesian}.
	
	\begin{proof}[Proof of Theorem~\ref{th:cartesian}]
		We consider the following two cases. 
		
		Case 1. $R_{y_1}(y_1, y_2) = \emptyset$. By definition
		\begin{align*}
			R_{(x, y_1)}((x, y_1), (x, y_2)) & = N_G(x)\times \{y_1\},
			\\
			R_{(x, y_2)}((x, y_1), (x, y_2)) & = N_G(x)\times \{y_2\}.
		\end{align*}
		It is easy to see that~$\text{OPT}_{G\Box H} = d_G$, and $\text{MAX}_{G\Box H} = 1$. 
		In addition, $|\bigtriangleup(y_1, y_2)| = d_H - 1$, and $\text{OPT}_H = 0$. 	
		By equations (\ref{eq:LLY}) and (\ref{eq:k02}), we have 
		\[
		\kappa_{\rm{LLY}}(y_1, y_2) = \frac{d_H + 1}{d_H},\quad \kappa_0(y_1, y_2) = \frac{d_H - 1}{d_H}.
		\] 	
		By equation (\ref{eq:LLY}),
		\begin{align*}
			\kappa_{\rm{LLY}}((x, y_1), (x, y_2)) & = \frac{1}{d_{G\Box H}}(d_{G\Box H} + 1 - d_G) \\
			& = \frac{d_H + 1}{d_{G\Box H}} \\
			& = \frac{d_H}{d_{G\Box H}}\kappa_{\rm{LLY}}(y_1, y_2).
		\end{align*}
		By equation (\ref{eq:k01}), 
		\begin{align*}
			\kappa_0((x, y_1), (x, y_2)) & = \frac{d_H + 1}{d_{G\Box H}} - \frac{1}{d_{G\Box H}}(3 - 1) \\
			& = \frac{d_H - 1}{d_{G\Box H}} \\
			& = \frac{d_H}{d_{G\Box H}}\kappa_0(y_1, y_2).
		\end{align*}
		Note that $|\bigtriangleup((x, y_1), (x, y_2))| < d_{G\Box H} - 1$ always holds in $G\Box H$. 
		
		Case 2. $R_{y_1}(y_1, y_2) \neq \emptyset$.  
		By Lemma \ref{le:infcartesian}, we have 
		\begin{align}
			\text{OPT}_{G\Box H} = \text{OPT}_H + d_G. 
			\label{eq:XGHcar}
		\end{align}
		Applying equation~(\ref{eq:LLY}) on~$H$, we get
		\begin{align}
			\text{OPT}_H = d_H - d_H\kappa_{\rm{LLY}}(y_1, y_2) + 1. 
			\label{eq:XHcar}
		\end{align}
		Combine equations~(\ref{eq:LLY}),~(\ref{eq:XGHcar}) and~(\ref{eq:XHcar}), 
		\begin{align}
			\kappa_{\rm{LLY}}((x, y_1), (x, y_2)) & = \frac{1}{d_{G\Box H}}(d_{G\Box H} + 1 - \text{OPT}_{G\Box H}) \nonumber \\
			& = \frac{d_H}{d_{G\Box H}}\kappa_{\rm{LLY}}(y_1, y_2). 
			\label{eq:LLYcar} 
		\end{align}
		By Lemma~\ref{le:supcartesian}, we have 
		\begin{align}
			\text{MAX}_{G\Box H} = \text{MAX}_H. 
			\label{eq:YGH=YHcar}
		\end{align}
		Applying equation~(\ref{eq:k01}) on~$H$, we get
		\begin{align}
			\text{MAX}_H = d_H(\kappa_0(y_1, y_2) - \kappa_{\rm{LLY}}(y_1, y_2)) + 3.
			\label{eq:YHcar}
		\end{align}
		Combine equations~(\ref{eq:k01}),~(\ref{eq:LLYcar}),~(\ref{eq:YGH=YHcar}), and~(\ref{eq:YHcar}),
		\begin{align*}
			\kappa_0((x, y_1), (x, y_2)) & = \frac{d_H}{d_{G\Box H}}\kappa_{\rm LLY}(y_1, y_2) - \frac{1}{d_{G\Box H}}(3 - \text{MAX}_{G\Box H})\\
			& = \frac{d_H}{d_{G\Box H}}\kappa_0(y_1, y_2). 
		\end{align*}
		This concludes the proof.
	\end{proof} 
	
	Finally, let us remark that throughout the paper, we focused on horizontal edges, the curvature results on vertical edges can be obtained by symmetry. 
	
	\bibliographystyle{plainurl}
	\bibliography{paper}
	
	\appendix
	
\end{document}